\documentclass[11pt,twoside]{article}
\usepackage[utf8]{inputenc}
\usepackage[english]{babel}
\usepackage{blindtext}
\usepackage{amssymb}
\usepackage{amsthm}
\usepackage{amsmath}
\usepackage{color}   
\usepackage{hyperref}
\usepackage{latexsym,amsfonts,epsfig,xcolor}
\usepackage{flexisym}
\usepackage{appendix}
\hypersetup{
    colorlinks=true, 
    linktoc=all,     
    linkcolor=blue,  
}
\newtheorem{theorem}{Theorem}[section]
\newtheorem{corollary}[theorem]{Corollary}

\newtheorem{lemma}[theorem]{Lemma}

\newcommand{\be}{\begin{equation}}
\newcommand{\ee}{\end{equation}}
\newcommand{\bea}{\begin{eqnarray*}}
\newcommand{\eea}{\end{eqnarray*}}
\newcommand{\bc}{\begin{center}}
\newcommand{\ec}{\end{center}}
\newcommand{\ba}{\begin{align}}
\newcommand{\ea}{\end{align}}
\newcommand{\by}{\begin{eqnarray}}
\newcommand{\ey}{\end{eqnarray}}
\newcommand{\hs}{\hspace{.3in}}

\newcommand{\CC}{\mathbb{C}}

\newcommand{\mnc}{M_n\,(\mathbb{C})}

\newfont{\bb}{msbm10}

\renewcommand{\l}{\langle}
\renewcommand{\r}{\rangle}

\date{}

\usepackage{caption}
\begin{document}
\title{Rank one perturbation with a generalized eigenvector}
\author{
Faith Zhang\\ \\
Department of Mathematics and Statistics\\
University of Massachusetts Amherst\\
Amherst, MA 01003\\
({\tt  yzhang@math.umass.edu})}
\maketitle

\begin{abstract}
The relationship between the Jordan structures of two matrices sufficiently close has been largely studied in the literature, among which a square matrix $A$ and its rank one updated matrix of the form $A+xb^*$ are of special interest. The eigenvalues of $A+xb^*$, where $x$ is an eigenvector of $A$ and $b$ is an arbitrary vector,  were first expressed in terms of eigenvalues of $A$ by Brauer in 1952. Jordan structures of $A$ and $A+xb^*$ have been studied, and similar results were obtained when a generalized eigenvector of $A$ was used instead of an eigenvector. However, in the latter case, restrictions on $b$ were put so that the spectrum of the updated matrix is the same as that of $A$. There does not seem to be results on the eigenvalues and generalized eigenvectors of $A+xb^*$ when $x$ is a generalized eigenvector and $b$ is an arbitrary vector. In this paper we show that the generalized eigenvectors of the updated matrix can be written in terms of those of $A$ when a generalized eigenvector of $A$ and  an arbitrary vector $b$ are involved in the perturbation.
\end{abstract}



\newpage
\section{Introduction}
In general, it is a very difficult problem to study the spectrum and the generalized eigenvectors of $A+xb^*$, given the spectrum and generalized eigenvectors of $A$,  when $x$  and $b$ are both arbitrary vectors. Note that $A+xb^*$ is a rank one updated matrix of $A$. \cite{T} showed how similarity invariants of matrices with elements in a field change under a rank one update. The perturbation theory of structured matrices under generic structured rank one perturbations was considered and generic Jordan structures of perturbed matrices were identified in \cite{MMRR}.
  
In 1945, \cite{B} described the relationship among eigenvalues of a given square matrix $A$ and $A+xb^*$, where $x$ is an eigenvector of $A$  and $b$ is an arbitrary vector. The result, known as the Brauer’s Theorem, reveals that $\bigg|\sigma(A+xb^*)\setminus\sigma(A)\bigg|\leq 1$, i.e., only one eigenvalue can be possibly updated during the perturbation. We will call $A+xb^*$ an eigenvector-updated (EU) matrix of $A$ when $x$ is an eigenvector of $A$. Brauer’s Theorem is in fact
related to older and well-known results on Wielandt’s and Hotelling’s deflations techniques \cite{H}.  

  The relations between the Jordan structures of a matrix and its EU matrix have been studied by \cite{BCU}, and the generalized eigenvectors of the updated matrix can be written in terms of the generalized eigenvectors of $A$. Similar results were discovered when $A+xb^*$ is obtained from a generalized eigenvector, that is, $x$ is a generalized eigenvector of $A$. We will call this $A+xb^*$ a generalized-eigenvector-updated (GEU) matrix of $A$. However, in the case of $x$ being a generalized eigenvector, restrictions on $b$ were put so that the spectrum of the updated matrix is the same as that of $A$. In this paper, we consider $b$ as an arbitrary vector and show that the generalized eigenvectors of the updated matrix can still be written in terms of the generalized eigenvectors of $A$, and the coefficients can be found recursively. 

Some applications of rank one perturbations with eigenvectors have been studied.  
Construction of nonnegative matrices with a prescribed spectrum was studied in \cite{PH}. 
The eigenvalue localization problem of control theory and stabilization of control systems were discussed in \cite{KT} . In \cite{Z} the authors worked on connections between pole assignment and assignment of invariant factors on matrices with some prescribed submatrices. 
The application to the nonnegative inverse eigenvalue problem can be followed in \cite{SR}. Some deflation problems and pole assignment for single-input single-output and multi-input multi-output systems were studied by \cite{BCSU},  to obtain a block version of deflation results,  using their result of the relationships between the right and left eigenvectors of a matrix and its EU matrix. The recursive construction of all matrices with positive principal minors (P-matrices) was studied in \cite{TZ}. These applications all motivate the study of rank one perturbations involving eigenvectors and generalized eigenvectors. 



\section{Notation and preliminaries}

Throughout this paper, we work with an $n\times n$ matrix $A$ in the set $\mnc$, and let $b\in\CC^n$ be an arbitrary $n$-dimensional vector. Denote the spectrum of $A$ by $\sigma(A)$ and the cardinality of a set $\alpha$ by $|\alpha|$. For a positive integer $m$, let $\l m\r=\{1,2,\ldots,m\}$.
 A vector $x_{m}$ is a generalized eigenvector of rank $m$ of $A$  corresponding to the eigenvalue $\lambda$ if
     $(A-\lambda I)^{m}x_{m}=0$ and  $(A-\lambda I)^{m-1}x_{m}\neq 0$. Clearly, a generalized eigenvector of rank 1 is an ordinary eigenvector. Let $J_r(\lambda)$ denote the standard $r \times r$ upper triangular Jordan block with eigenvalues $\lambda$:
\begin{center}
$J_r(\lambda)=\begin{bmatrix}
\lambda&1&0&\cdots&0\\0&\ddots&\ddots&\ddots&\vdots\\ \vdots&\ddots&\ddots&\ddots&0\\ \vdots&&\ddots&\lambda&1\\0&\cdots&\cdots&0&\lambda
\end{bmatrix}_{r \times r}$.
\end{center}
For notational simplicity, we make the following assumptions: \begin{itemize}
	\item Whenever a subscript or superscript is $0$, the corresponding scalar or vector is $0$ by default.
	\item When a result is in fraction form, we assume that we have only considered the case where the denominator is nonzero. 
\end{itemize}
Therefore, generalized eigenvectors of $A$ in a Jordan chain of $\lambda$ satisfy $Ax_m=\lambda x_m+x_{m-1}$ ($m$ can be 1 by the first assumption listed above). 
Our main effort will be devoted into finding generalized eigenvectors of a GEU matrix $A+x_mb^*$. 
We first include the following well-known Matrix Determinant Lemma and a proof for completeness.
\begin{lemma} 
	\label{MDL}
	Let $A\in\mnc$ be invertible and $x, b\in\CC^n$. Then 
$$\det \left( A + x b^* \right) = \left(b^* A^{ - 1} x+1 \right) \,\det (A) \ .$$
\end{lemma}
\begin{proof}$\;$
	Since 
	\begin{eqnarray*}
		&&\begin{bmatrix}I&0\\b^*&1 \end{bmatrix}
		\begin{bmatrix}I+xb^*&x\\0&1 \end{bmatrix}
		\begin{bmatrix}I&0\\b^*&1 \end{bmatrix}^{-1} \\&=& 
		\begin{bmatrix}I&0\\b^*&1 \end{bmatrix}
		\begin{bmatrix}I+xb^*&x\\0&1 \end{bmatrix}
		\begin{bmatrix}I&0\\-b^*&1 \end{bmatrix}  = \begin{bmatrix}I&x\\0&b^*x+1 \end{bmatrix}, 
	\end{eqnarray*}
	we have

	\[
	\det\begin{bmatrix}I+xb^*&x\\0&1 \end{bmatrix} = 
	\det\begin{bmatrix}I&x\\0&b^* x+1 \end{bmatrix},
	\] 
	i.e., $\det(I+xb^*)=b^*x+1$. Thus,
	\[
	\det \left(A + x b^*\right) = \det(A)\,\det\left(I+A^{-1}xb^*\right) = \left(b^*A^{-1}x+1\right) \det (A).
	\]
\end{proof}
The rest of the paper includes the main results of eigenvalues and generalized eigenvectors of a GEU matrix. 
\section{Perturbation via a generalized eigenvector}
In what follows, we use the following notation. Let $A \in \mnc$ and $\lambda$ be an eigenvalue of $A$ with a Jordan block $J_r(\lambda)$ of size $r$. For a positive integer $m\leq r$, let $x_m$ denote a generalized eigenvector of rank $m$ of $A$ associated with $\lambda$. 

\begin{theorem}\label{thm21}
 $\bigg|\sigma(A+x_mb^*)\setminus\sigma(A)\bigg|\leq m$, i.e., at most $m$ eigenvalues of $A+x_mb^*$ are different from those of $A$, with the $m$ eigenvalues being zeros of 
  \be f(t)=(t-\lambda)^m-b^*\sum\limits_{i=0}^{m-1}x_{i+1}(t-\lambda)^i\ .\label{ft}\ee
\end{theorem}

\begin{proof}$\;$ 
We first prove by induction the claim that \be (tI-A)^{-1}x_m=\sum\limits_{i=0}^{m-1}\frac{x_{i+1}}{(t-\lambda)^{m-i}}\ . \label{21}\ee
When $m=1$, $x_1$ is an eigenvector of $A$ associated with $\lambda$, then $(tI-A)x_1=(t-\lambda)x_1\ .$ Multiplying both sides by $\frac{(tI-A)^{-1}}{t-\lambda}$, we have that for $t\notin\sigma(A)$,
$$\frac{x_1}{t-\lambda}=(tI-A)^{-1}x_1\ .$$ Suppose that (\ref{21}) holds for $m=k$, when $m=k+1$, the equality
 $$(tI-A)x_{k+1}=tx_{k+1}-Ax_{k+1}=tx_{k+1}-\lambda x_{k+1}-x_{k}=(t-\lambda) x_{k+1}-x_{k}$$  yields $$\frac{x_{k+1}}{t-\lambda}=(tI-A)^{-1}x_{k+1}-\frac{(tI-A)^{-1}x_{k}}{t-\lambda}\ ,$$ i.e., $$(tI-A)^{-1}x_{k+1}=\frac{x_{k+1}}{t-\lambda}+\frac{(tI-A)^{-1}x_{k}}{t-\lambda}=\sum\limits_{i=0}^{k}\frac{x_{i+1}}{(t-\lambda)^{k+1-i}}\ $$ by the induction hypothesis, thus have shown the claim. Now for all $t\notin\sigma(A)$, by Lemma \ref{MDL}, 
\begin{eqnarray}
	\det\left(tI-\left(A+x_mb^*\right)\right)&=&\left(-b^*\left(tI-A\right)^{-1}x_m+1\right)\det(tI-A )\nonumber\\
	&=&\left(-b^*\sum\limits_{i=0}^{m-1}\frac{x_{i+1}}{(t-\lambda)^{m-i}}+1\right)\prod_j(t-\lambda_j)\ ,                       
\label{det}\end{eqnarray}
where $\lambda_j$'s are eigenvalues of $A$. 
Therefore, $m$ identical factors of $(t-\lambda)$ in $\prod_j(t-\lambda_j)$ cancel 
$(t-\lambda)^m$ in the denominator and the result follows.
\end{proof}
The result is in fact an extension of Brauer's Theorem (see \cite{B}). When $m=1$, $A+x_1b^*$ is an EU matrix of $A$ and (\ref{ft}) reduces to \be f(t)=(t-\lambda)-b^*x_{1}\ ,\nonumber\ee
which has a zero $t=\lambda+b^*x_{1}$ and agrees with the result in \cite{B}.

\begin{corollary}
In Theorem $\ref{thm21}$, if $b$ satisfies $b^*x_j=0$ for all $j\in \l m\r$, then all eigenvalues are preserved, so are generalized eigenvectors associated with $\lambda$, up to rank $m$.
\end{corollary}
\begin{proof}$\;$ 
All eigenvalues are preserved since $$\det\left(tI-\left(A+x_mb^*\right)\right)=\prod_j(t-\lambda_j)$$ by (\ref{det}). Now that $\left(A+x_mb^*\right)x_j=Ax_j=\lambda x_j+x_{j-1}$ for $j\in \l m\r$, generalized eigenvectors  associated with $\lambda$ up to rank $m$ are also preserved.
\end{proof}

After a perturbation, we analyze generalized eigenvectors corresponding to three types of preserved eigenvalues: 
\begin{itemize}
	\item[(i)] An unchanged eigenvalue $\lambda$ in $J_r(\lambda).$
	\item[(ii)] A common eigenvalue $\lambda$ of $A$ and $A+x_mb^*$ in some other Jordan block, say $J_s(\lambda)$. 
	\item[(iii)] A common eigenvalue of $A$ and $A+x_mb^*$ that is not equal to $\lambda$. 
\end{itemize}
We next consider all three cases in the following sections and  give recusive formulae for generalized eigenvectors of $A+x_mb^*$ in terms of generalized eigenvectors of $A$.

\subsection{Eigenvalue $\lambda$ remaining in the block $J_r(\lambda)$}
After a perturbation, it is possible that some $\lambda$'s in $J_r(\lambda)$ remain unchanged (e.g., $m<r$).  We now give another sufficient condition.
\begin{theorem}
  In Theorem $\ref{thm21}$, if $b$ is such that $b^*x_j=0$ for $j=1,..., k\leq m,$ then $\bigg|\sigma(A+x_mb^*)\setminus\sigma(A)\bigg|\leq m-k$, i.e., at most $m-k$ eigenvalues of $A+x_mb^*$ are different from those of $A$.
\end{theorem}
\begin{proof}$\;$ Since
  $$ f(t)=(t-\lambda)^k\left((t-\lambda)^{m-k}-b^*\sum\limits_{i=k}^{m-1}x_{i+1}(t-\lambda)^{i-k}\right)\ $$ by (\ref{ft}), the result follows.
\end{proof}
This result is a generalization of Theorem 5 in \cite{BCU}, which can be seen by letting $k=m-1$ above. Then (\ref{ft}) leads to  \be f(t)=(t-\lambda)^{m-1}(t-\lambda-b^*x_{m})\nonumber\ \ee and results in the three cases detailed in \cite{BCU}. An alternative description of the above result is that, at least $r-m+k$ eigenvalues are preserved in that same Jordan block. We next explore their corresponding generalized eigenvectors of $A+x_mb^*$.
\begin{theorem}\label{thm34}
   If $\lambda$ now has a Jordan block of size $\tilde{r}$, then for $t\leq \tilde{r}$ and $m+t\leq r$, $A+x_mb^*$ has a generalized eigenvector of rank $t$ associated with $\lambda$ given by
  
  $(1)$ when $t \leq m+1$, \be u_t=x_t+\sum\limits_{j=1}^{t-1} \beta_{j}^{(t)}x_j+\beta x_{m+t}\ ,\label{22eq24}\ee where $ \beta=\frac{-b^*x_1}{1+b^*x_{m+1}}$ and the coefficients $\beta_{j}^{(t)}$ satisfy a recurrence relation \be \beta_{j+1}^{(t)}=\beta_{j}^{(t-1)}\ , \hs j\in \l t-2\r \label{22eq25}\ee  and \be \beta_{1}^{(t)}= \frac{-b^*x_t-\sum\limits_{j=2}^{t-1}\beta_{j-1}^{(t-1)}(b^*x_j)-\beta b^*x_{m+t}}{b^*x_1}\ , \hs t\geq 2 \ ;\label{22eq26}\ee 
  
  $(2)$ when $t > m+1$ $($hence $t>1)$, \be u_t=x_t+\sum\limits_{j=1}^m \beta_{j}^{(t)}x_j+\beta x_{m+t}\ ,\label{22eq27}\ee where the coefficients $\beta_{j}^{(t)}$ satisfy the same recurrence relation as in $(\ref{22eq25})$ for $j\in \l m-1\r$, $\beta$ is as given in part $(1)$ and \be \beta_{1}^{(t)}= \frac{\beta_{m}^{(t-1)}-b^*x_t-\sum\limits_{j=2}^{m}\beta_{j-1}^{(t-1)}(b^*x_j)-\beta b^*x_{m+t}}{b^*x_1} \ .\label{22eq28}\ee

  \end{theorem}
\begin{proof}$\;$ 
We show that $$(A+x_mb^*)u_t=\lambda u_t+u_{t-1}\ $$ for $t \leq m+1$ 
and the case of $t> m+1$ is similar. 
Since \bea L.H.S.&= &(A+x_mb^*)\left(x_t+\sum\limits_{j=1}^{t-1} \beta_{j}^{(t)}x_j+\beta x_{m+t}\right)
\\&=&
\lambda x_t+x_{t-1}+\sum\limits_{j=1}^{t-1} \beta_{j}^{(t)}\lambda x_j+\sum\limits_{j=1}^{t-2} \beta_{j+1}^{(t)} x_j+\beta \lambda x_{m+t}+\beta x_{m+t-1}\\ 
&\hs&+\ b^*\left(x_t+\sum\limits_{j=1}^{t-1} \beta_{j}^{(t)}x_j+\beta x_{m+t}\right)x_m\ 
\eea 
and 
\bea R.H.S.&=&
\lambda \left(x_t+\sum\limits_{j=1}^{t-1} \beta_{j}^{(t)}x_j+\beta x_{m+t}\right)+\left(x_{t-1}+\sum\limits_{j=1}^{t-2} \beta_{j}^{(t-1)}x_j+\beta x_{m+t-1}\right)\ ,
\eea
the result follows from (\ref{22eq25}) and (\ref{22eq26}) .
\end{proof}
\begin{corollary}
$A+x_mb^*$ has an eigenvector associated with $\lambda$ given by
   $$ u_1=x_1-\frac{b^*x_1}{1+b^*x_{m+1}} x_{m+1}\ ;$$
In particular,
   $A+x_1b^*$ has an eigenvector associated with $\lambda$ given by
   \be u_1=x_1-\frac{b^*x_1}{1+b^*x_{2}} x_{2}\ .\label{2nd}\ee
\end{corollary}
The expression (\ref{2nd}) offers insight into calculating eigenvectors of the EU matrix $A+x_1b^*$ and hence generalizes the results in \cite{Y, BCSU}. It is also a scalar multiple of the expression in Theorem 2 part (2.2) of \cite{BCU} and thus viewed as identical. We will see in the following section that there are other generic expressions for eigenvectors of $A+x_1b^*$ when $A$ has more than one Jordan block.
\subsection{A common eigenvalue $\lambda$ in a different Jordan block}
If besides a Jordan block $J_r(\lambda)$ of size $r$, $A$ has another Jordan block $J_s(\lambda)$ of size $s$ corresponding to $\lambda$, what are the generalized eigenvectors? For positive integers $m\leq r$ and $t\leq s$, let $x_m$ and $y_t$ denote generalized eigenvectors of rank $m$ and $t$ associated with $\lambda$ in two blocks respectively.
\begin{theorem}\label{thm36}
  $A+x_mb^*$ has a generalized eigenvector of rank $t$ associated with $\lambda$ given by 
 
 $(1)$ when $t \leq m$, \be v_t=y_t+\sum\limits_{j=1}^t \beta_{j}^{(t)}x_j\ ,\label{23eq24}\ee where the coefficients $\beta_{j}^{(t)}$ satisfy a recurrence relation \be \beta_{j+1}^{(t)}=\beta_{j}^{(t-1)}\ , \hs j\in \l t-1\r \label{23eq25}\ee and \be \beta_{1}^{(t)}= \frac{-b^*y_t-\sum\limits_{j=2}^{t}\beta_{j-1}^{(t-1)}(b^*x_j) }{b^*x_1} \ ;\label{23eq26}\ee
  $(2)$ when $t > m$ $($hence $t>1)$, \be v_t=y_t+\sum\limits_{j=1}^m \beta_{j}^{(t)}x_j\ ,\label{23eq27}\ee where the coefficients $\beta_{j}^{(t)}$ satisfy the same recurrence relation as $(\ref{23eq25})$ for $j\in \l m-1\r$ and \be \beta_{1}^{(t)}= \frac{\beta_{m}^{(t-1)}-b^*y_t-\sum\limits_{j=2}^{m}\beta_{j-1}^{(t-1)}(b^*x_j) }{b^*x_1} \ .\label{23eq28}\ee
 
  \end{theorem}
\begin{proof}$\;$ 
We show that $$(A+x_mb^*)v_t=\lambda v_t+v_{t-1}\ $$ for $t \leq m$ 
and the case of $t> m$ is similar. 
Since \bea L.H.S.&=&(A+x_mb^*)\left(y_t+\sum\limits_{j=1}^t \beta_{j}^{(t)}x_j\right)
\\&=&
\lambda y_t+y_{t-1}+\sum\limits_{j=1}^{t} \beta_{j}^{(t)}\lambda x_j+\sum\limits_{j=1}^{t-1} \beta_{j+1}^{(t)} x_j+b^*\left(y_t+\sum\limits_{j=1}^t \beta_{j}^{(t)}x_j\right)x_m\ 
\eea 
and 
\bea R.H.S.&=&
\lambda\left(y_t+\sum\limits_{j=1}^t \beta_{j}^{(t)}x_j\right)+\left(y_{t-1}+\sum\limits_{j=1}^{t-1} \beta_{j}^{(t-1)}x_j\right)\ ,
\eea
the result follows from (\ref{23eq25}) and (\ref{23eq26}) .
\end{proof}
\begin{corollary}
    $A+x_mb^*$ $(m\leq r)$ share an eigenvector associated with $\lambda$ given by 
 $$ v_1=y_1 -\frac{b^*y_1}{b^*x_1} x_1\ .$$
\end{corollary}
Note that when $m=1$, the above result is a special case of that in \cite{Y, BCSU} where both eigenvalues involved are equal to $\lambda$. It also agrees with the case $j=1$ in Theorem 2 part (1) of \cite{BCU}. Furthermore, $v_1$ is also an eigenvector of $A$ associated with $\lambda$. Therefore, all GEU matrices as such share an eigenvector with $A$.
\subsection{A common eigenvalue that is different than $\lambda$}
All eigenvalues of $A$ not equal to $\lambda$ will remain to be eigenvalues of $A+x_mb^*$.
\begin{theorem}\label{thm38}
 Suppose an eigenvalue $\mu\neq\lambda$ of $A$ has a Jordan block of size $\gamma$ and let $z_t$ be an associated generalized eigenvector of rank $t\leq \gamma$. Then $A+x_mb^*$ has a generalized eigenvector of rank $t$ associated with $\mu$, given by \be w_t=z_t+\sum\limits_{j=1}^m \beta_{j}^{(t)}x_j\ ,\label{eq24}\ee where the coefficients $\beta_{j}^{(t)}$ satisfy a recurrence relation \be \beta_{j}^{(t)}=(\mu-\lambda)^{j-m}\left(\beta_{m}^{(t)}-\sum\limits_{i=0}^{m-1-j}(\mu-\lambda)^{i}\beta_{m-1-i}^{(t-1)}\right)\ , \hs j\in \l m-1\r \label{eq25}\ee  with \be \beta_{m}^{(t)}=\frac{(\mu-\lambda)^{m}\left(b^*z_t-\beta_{m}^{(t-1)}\right)-\sum\limits_{j=1}^{m}(b^*x_j)\sum\limits_{i=0}^{m-1-j} (\mu-\lambda)^{i+j} \beta_{m-1-i}^{(t-1)}}{(\mu-\lambda)^{m+1}-\sum\limits_{j=1}^m (\mu-\lambda)^{j}(b^*x_j)}\ .\label{eq26}\ee
\end{theorem}
\begin{proof}$\;$ 
We show that $$(A+x_mb^*)w_t=\mu w_t+w_{t-1}\ .$$ 
Since \bea L.H.S.&= &(A+x_mb^*)\left(z_t+\sum\limits_{j=1}^m \beta_{j}^{(t)}x_j\right)
\\&=&
\mu z_t+z_{t-1}+\sum\limits_{j=1}^{m-1} \left(\beta_{j}^{(t)}\lambda+\beta_{j+1}^{(t)}\right) x_j+\left(b^*z_t+\beta_{m}^{(t)}\lambda+\sum\limits_{j=1}^m \beta_{j}^{(t)}\left(b^*x_j\right)\right)x_m\ 
\eea 
and 
\bea R.H.S.&=&
\mu z_t+z_{t-1}+\sum\limits_{j=1}^{m} \left(\beta_{j}^{(t)}\mu+\beta_{j}^{(t-1)}\right) x_j\ ,
\eea
it remains to show that \be\beta_{j}^{(t)}\lambda+\beta_{j+1}^{(t)}=\beta_{j}^{(t)}\mu+\beta_{j}^{(t-1)}\ , \hs j\in \l m-1\r \label{eq27}\ee and \be b^*z_t+\beta_{m}^{(t)}\lambda+\sum\limits_{j=1}^m \beta_{j}^{(t)}\left(b^*x_j\right)=\beta_{m}^{(t)}\mu+\beta_{m}^{(t-1)}\ .\label{eq28}\ee
One can easily verify that (\ref{eq27}) follows directly from (\ref{eq25}), while (\ref{eq28}) follows from (\ref{eq25}) and (\ref{eq26}) .
\end{proof}
\begin{corollary}
   If $z_1$ is an eigenvector of $A$ associated with $\mu\neq\lambda$, then an eigenvector of $A+x_mb^*$ associated with $\mu$ is \be w_1=z_1+\frac{b^*z_1}{(\mu-\lambda)^{m+1}-\sum\limits_{j=1}^m (\mu-\lambda)^{j}(b^*x_j)}\sum\limits_{j=1}^m (\mu-\lambda)^{j}x_j\ .\nonumber\ee 
\end{corollary}

Note that when $m=1$, an eigenvector of $A+x_1b^*$ associated with $\mu$ is \be w_1=z_1 +\frac{b^*z_1}{\mu-\lambda -b^*x_1} x_1\ ,\nonumber\ee which has been proved in \cite{Y, BCSU, BCU}.

\subsection{A comprehensive example}
We provide an example in this section to show the computation of all cases discussed previously. It is a real matrix example for calculation simplicity, even though the results work for complex matrices in general. Let $e_i$ denote the $i$-th column of an $11\times 11$ identity matrix, i.e., the $i^{th}$ standard basis vector.

Let $$A=J_6(2)\oplus J_3(2)\oplus J_2(1)$$ and a generalized eigenvector $x_2$ corresponding to 2 in $J_6(2)$ be used to obtain $A+x_2b^*$. Thus $\lambda=2, \mu=1, r=6, s=3$ and $m=2$. Select a  vector $b=3e_1-5e_2+e_7+e_{10}$ (randomly) and apply results in previous sections to find generalized eigenvectors of our interest for $A+x_2b^*$. To illustrate the effectiveness of the theorems, we calculate a generalized eigenvector for each of the following: Theorem $\ref{thm34}$ $(1)$ and $(2)$, Theorem $\ref{thm36}$ $(1)$ and $(2)$, as well as Theorem $\ref{thm38}$. 

Clearly $$x_j=\sum\limits_{i=1}^{j}e_i\ (j=1,...,6),\ y_{j-6}=\sum\limits_{i=7}^{j}e_i\ (j=7,8,9),\ z_{j-9}=\sum\limits_{i=10}^{j}e_i\ (j=10,11)$$ are generalized eigenvectors corresponding to 2 in $J_6(2)$, 2 in $J_3(2)$ and 1 in $J_2(1)$ respectively. One can easily verify by doing a Jordan decomposition that, after the perturbation, the Jordan form of $A+x_2b^*$ is $$J=J_4(2)\oplus J_3(2)\oplus J_2(1)\oplus J_1(3)\oplus J_1(-1)\ ,$$ thus only the first Jordan block of $A$ is lost and the new eigenvalues 3 and -1 of $A+x_2b^*$ are zeros of 
$$ f(t)=(t-2)^2-3-(-2)(t-2)$$
as indicated by Theorem \ref{thm21}. We now consider the first three blocks of $J$ that contain the preserved eigenvalues.
\begin{itemize}
\item  For $J_4(2)$:
\begin{itemize}
    \item [(i)] $u_3=x_3+\frac{40}{9}x_1+\frac{8}{3}x_2+3x_5$ by (\ref{22eq24})
    \item [(ii)] $u_4=x_4+\frac{176}{27}x_1+\frac{40}{9}x_2+3x_6$ by (\ref{22eq27})
    
    where the coefficients were obtained in the following order: $\beta=3$, $\beta_{2}^{(3)}=\beta_{1}^{(2)}=\frac{8}{3}$, $\beta_{2}^{(4)}=\beta_{1}^{(3)}=\frac{40}{9}$  by (\ref{22eq25}) and (\ref{22eq26}), and $\beta_{1}^{(4)}=\frac{176}{27}$ by (\ref{22eq28});
\end{itemize}
\item  For $J_3(2)$:
\begin{itemize}
    \item [(i)] $v_2=y_2-\frac{5}{9}x_1-\frac{1}{3}x_2$ by (\ref{23eq24})
    \item [(ii)] $v_3=y_3-\frac{22}{27}x_1-\frac{5}{9}x_2$ by (\ref{23eq27})
    
    where the coefficients were obtained in the following order: $\beta_{2}^{(2)}=\beta_{1}^{(1)}=-\frac{1}{3}$, $\beta_{2}^{(3)}=\beta_{1}^{(2)}=-\frac{5}{9}$  by (\ref{23eq25}) and (\ref{23eq26}), and $\beta_{1}^{(3)}=-\frac{22}{27}$ by (\ref{23eq28});
\end{itemize}
\item  For $J_2(1)$: \hs $w_2=z_2-\frac{1}{4}x_1$ by (\ref{eq24})

where the coefficients were calculated in the order of $\beta_{2}^{(1)}=\frac{1}{4}$, $\beta_{1}^{(1)}=-\frac{1}{4}$, $\beta_{2}^{(2)}=0$ and $\beta_{1}^{(2)}=-\frac{1}{4}$  by (\ref{eq25}) and (\ref{eq26}).
\end{itemize}
One can verify that the vectors obtained above are indeed generalized eigenvectors of $A+x_2b^*$ of corresponding ranks.
\section{Discussion}
We considered a rank one perturbation $A+xb^*$ when $x$ is a generalized eigenvector and $b$ is an arbitrary vector. As a result, we gave a function whose zeros are the updated eigenvalues of a GEU matrix $A+xb^*$, though the zeros can not be specified explicitly in general yet. We also provided the generalized eigenvectors of $A+xb^*$ as a linear combination of those of $A$, whose coefficients can be found recursively. 

It remains to be an open problem how to express the generalized eigenvectors of $A+xb^*$ associated with the brand new eigenvalues after a perturbation. There seems to be none available results even for the case of $x$ being an ordinary eigenvector. There are wide applications of EU matrices, however, applications of the GEU type might be in broader areas, some of which are still to be explored. 

\begin{appendices}
\end{appendices}
\end{document}